\documentclass{amsart}
\usepackage[margin=2.7cm]{geometry}
\usepackage[parfill]{parskip}   
\usepackage{amsthm}
\usepackage{amssymb}
\usepackage{amsmath}
\usepackage{amsfonts}
\usepackage{url}
\usepackage[dvips]{graphicx,color}
\usepackage{psfrag}
\usepackage{etex}
\usepackage[all]{xy} 
\usepackage{subfig}
\usepackage{listings}
\usepackage{array}

\theoremstyle{plain}
\newtheorem{theorem}{Theorem}

\newtheorem{conjecture}[theorem]{Conjecture}
\newtheorem{corollary}[theorem]{Corollary}

\newtheorem{lemma}[theorem]{Lemma}
\newtheorem{proposition}[theorem]{Proposition}

\newtheorem{question}[theorem]{Question}

\newcommand{\Qsub}[1]{\Q_{(#1)}}
\newcommand{\one}{\mathbf{1}}

\newcommand{\F}{\mathbb{F}}
\newcommand{\Z}{\mathbb{Z}}

\newcommand{\Q}{\mathbb{Q}}

\newcommand{\os}{Ozsv\'ath and Szab\'o}

\newcommand{\isom}{\cong}
\newcommand{\goto}{\rightarrow}

\newcommand{\x}{\mathbf{x}}
\newcommand{\y}{\mathbf{y}}

\newcommand{\rin}{\reflectbox{\rotatebox[origin=c]{90}{$\in$}}}

\newcommand{\kh}{\operatorname{Kh}}
\newcommand{\unkh}{\widehat{\operatorname{Kh}}}
\newcommand{\hfk}{\widehat{\operatorname{HFK}}}
\newcommand{\hf}{\widehat{\operatorname{HF}}}
\newcommand{\cfk}{\widehat{\operatorname{CFK}}}

\newcommand{\hfkm}{\operatorname{HFK^{--}}}

\newcommand{\cfkm}{\operatorname{CFK^{--}}}
\newcommand{\cflm}{\operatorname{CFL^{--}}}

\DeclareMathOperator{\rank}{rank}

\DeclareMathOperator{\Ker}{Ker}

\title{Genus two mutant knots with the same dimension in knot Floer and Khovanov homologies}
\author[Moore]{Allison Moore}
\address{Department of Mathematics \\
          The University of Texas \\
          Austin, TX 78712, USA}
\email{moorea8@math.utexas.edu}
\urladdr{http://www.ma.utexas.edu/$\sim$moorea8}

\author[Starkston]{Laura Starkston}
\address{Department of Mathematics \\
          The University of Texas \\
          Austin, TX 78712, USA}
\email{lstarkston@math.utexas.edu}
\urladdr{http://www.ma.utexas.edu/$\sim$lstarkston}

\date{}                                           

\begin{document}

\begin{abstract}
We exhibit an infinite family of knots with isomorphic knot Heegaard Floer homology. Each knot in this infinite family admits a nontrivial genus two mutant which shares the same total dimension in both knot Floer homology and Khovanov homology. Each knot is distinguished from its genus two mutant by both knot Floer homology and Khovanov homology as bigraded groups. Additionally, for both knot Heegaard Floer homology and Khovanov homology, the genus two mutation interchanges the groups in $\delta$-gradings $k$ and $-k$. 
\end{abstract} 

\maketitle

\section{Introduction}
\label{sec:Introduction}

Genus two mutation is an operation on a three-manifold $M$ in which an embedded, genus two surface $F$ is cut from $M$ and reglued via the hyperelliptic involution $\tau$. The resulting manifold is denoted $M^\tau$. When $M$ is the three-sphere, the genus two mutant manifold $(S^3)^\tau$ is homeomorphic to $S^3$ (see Section~\ref{sec:Mutation}). If $K\subset S^3$ is a knot disjoint from $F$, then the knot that results from performing a genus two mutation of $S^3$ along $F$ is denoted $K^\tau$ and is called a {\it genus two mutant of the knot} $K$. The related operation of Conway mutation in a knot diagram can be realized as a genus two mutation or a composition of two genus two mutations (see Section~\ref{sec:Mutation}). 

In~\cite{OS:GenusBounds}, \os~ demonstrate that as a bigraded object, knot Heegaard Floer homology can detect Conway mutation. However, it can be observed that in all known examples~\cite{BG:Computations}, the rank of $\hfk(K)$ as an ungraded object remains invariant under Conway mutation. The question of whether the rank of knot Floer homology is unchanged under Conway mutation, or  more generally, genus two mutation, remains an interesting open problem. Moreover, while it is known that Khovanov homology with $\F_2=\Z/2\Z$ coefficients is invariant under Conway mutation~\cite{Bloom},\cite{Wehrli:Mutation}, the case of $\Z$-coefficients is also unknown. The invariance of the rank of Khovanov homology under genus two mutation constitutes a natural generalization of the question.  Recently, Baldwin and Levine have conjectured~\cite{BL:Spanning} that the $\delta$-graded knot Floer homology groups
\[
	\hfk_{\delta}(L) = \bigoplus_{\delta=a-m}\hfk_m(L,a)
\]
are unchanged by Conway mutation, which implies that their total ranks are preserved, amongst other things. A parallel conjecture can be made about $\delta$-graded Khovanov homology, and the $\delta$-graded Khovanov homology groups are given by
\[  
	\kh_{\delta}(L) = \bigoplus_{\delta=q-2i}\kh_q^i(L).
\]
In this note, we offer an example of an infinite family of knots with isomorphic knot Floer homology, all of which admit a genus two mutant of the same dimension in both $\hfk$ and $\kh$, though each pair is distinguished by both $\hfk$ and $\kh$ as bigraded vector spaces.\footnote{Because we compute $\hfk$ and $\kh$ as graded vector spaces over $\Z/2\Z$ or $\Q$, the theorem has been formulated in terms of dimension rather than rank. }  Additionally, we show that both the $\delta$-graded $\hfk$ and $\kh$ groups distinguish the genus two mutants pairs. Here, knot Floer homology computations are done with $\F_2$-coefficients, and Khovanov homology computations are done with $\Q$-coefficients.

\begin{theorem}
\label{thm:main}
There exists an infinite family of genus two mutant pairs $(K_n, K_n^\tau)$, $n\in\Z^+$, in which
\begin{enumerate}
	\item each infinite family has isomorphic knot Floer homology groups, \begin{eqnarray*}
		\hfk_m(K_n,a) &\cong& \hfk_m(K_0,a), \;\text{for all } m,a \\
		\hfk_m(K_n^\tau,a) &\cong& \hfk_m(K_0^\tau,a), \; \text{for all } m,a,
	\end{eqnarray*}
	\item each genus two mutant pair shares the same total dimension in $\hfk$ and $\kh$, \begin{eqnarray*}
		\bigoplus_{m,a} \dim_{\F_2}\;\hfk_m(K_n,a) &=& \bigoplus_{m,a} \dim_{\F_2}\;\hfk_m(K_n^\tau,a) \\
		\bigoplus_{i,q} \dim_{\Q}\;\kh^i_q (K_n) &=& \bigoplus_{i,q} \dim_{\Q}\;\kh^i_q (K_n^\tau),
	\end{eqnarray*}
	\item each genus two mutant pair is distinguished by $\hfk$ and $\kh$ as bigraded groups, \begin{eqnarray*}
		\hfk_m(K_n,a) &\not\cong& \hfk_m(K_n^\tau,a) \; \text{for some }m,a\\
		\kh^i_q(K_n) &\not\cong& \kh_q^i(K_n^\tau)  \; \text{for some }i,q,
	\end{eqnarray*}
	\item each genus two mutant pair is distinguished by $\delta$-graded $\hfk$ and $\delta$-graded $\kh$, and moreover \begin{eqnarray*}
		\hfk_\delta(K_n) &\cong& \hfk_{-\delta}(K_n^\tau) \; \text{for all } \delta\\
		\kh_\delta(K_n) & \cong& \kh_{-\delta}(K_n^\tau) \; \text{for all } \delta. \\
	\end{eqnarray*}
\end{enumerate}
\end{theorem}
This example suggests that having invariant dimension of knot Floer homology or Khovanov homology is a property shared not only by Conway mutants, but by genus two mutant knots as well, offering positive evidence towards all the above open questions about total rank.

\subsection{Organization}In Section~\ref{sec:Mutation} we review genus two mutation and describe the infinite family of genus two mutant pairs. In Section~\ref{sec:Floer} we show that within each infinite family $\{K_n\}$ and $\{K^\tau_n\}$, the knots have isomorphic knot Heegaard Floer homology and that these families share the same dimension. In Section~\ref{sec:Khovanov} we show that each family also shares the same dimension of Khovanov homology. In Section~\ref{sec:Observation} we mention a few observations.

\section{Genus Two Mutation}
\label{sec:Mutation}

\begin{figure}[htb]
	\centerline{
	\includegraphics[width=9cm]{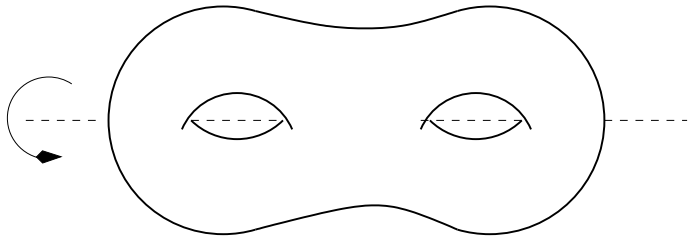}
	}
 	\caption{The genus two surface $F$ and hyperelliptic involution $\tau$.}
    	\label{fig:genus2surface}
\end{figure}
  
Let $F$ be an embedded, genus two surface in a compact, orientable three-manifold $M$, equipped with the hyperelliptic involution $\tau$. A genus two mutant of $M$, denoted $M^\tau$, is obtained by cutting $M$ along $F$ and regluing the two copies of $F$ via $\tau$~\cite{DGST:Behavior}. The involution $\tau$ has the property that an unoriented simple closed curve $\gamma$ on $F$ is isotopic to its image $\tau(\gamma)$.

When $M=S^3$, any closed surface $F\subset S^3$ is compressible. This implies by the Loop Theorem that $(S^3)^\tau$ is homeomorphic to $S^3$~\cite{DGST:Behavior}. Therefore, if $S^3$ contains a knot $K$ disjoint from $F$, mutation along $F$ is a well-defined homeomorphism of $S^3$ taking a knot $K$ to a potentially different knot $K^\tau$~\cite{DGST:Behavior}. In this note, we restrict our attention to surfaces of mutation which bound a handlebody containing $K$ in its interior. These mutations are called \emph{handlebody mutations}.

A Conway mutant of a knot $K\subset S^3$ is similarly obtained by an operation under which a Conway sphere $S$ interests $K$ in four points and bounds a ball containing a tangle. The ball containing the tangle is replaced by its image under a rotation by $\pi$ about a coordinate axis. In fact, Conway mutation of a knot can be realized as a special case of genus two mutation. Since $S$ separates $K$ into two tangles, i.e.
\[
	K = T_1 \cup_S T_2
\]
a genus two surface $F$ is formed by taking $S$ and tubing along either $T_1$ or $T_2$. The Conway mutation is then achieved by performing at most two such genus two mutations~\cite{DGST:Behavior}. Like Conway mutants, genus two mutants are difficult to detect and are indistinguishable by many knot invariants~\cite{DGST:Behavior}.

\begin{theorem}~\cite{CL:Mutations},~\cite{MT:Jones}
\label{thm:AlexanderJones}
The Alexander polynomial and colored Jones polynomials for all colors of a knot in $S^3$ are invariant under genus two mutation. Generalized signature is invariant under genus two handlebody mutation.
\end{theorem}

\begin{theorem}~\cite[Theorem 1.3]{Ruberman:Volumes}
\label{thm:Volume}
Let $K^\tau$ be a genus two mutation of the hyperbolic knot $K$. Then $K^\tau$ is also hyperbolic, and the volumes of their complements are the same.
\end{theorem}
Theorem~\ref{thm:Volume} is a special case of a more general theorem which shows that the Gromov norm is preserved under mutation along any of several symmetric surfaces, including the genus two surface on which we are focused here. Ruberman also shows that cyclic branched coverings and Dehn surgeries along a Conway mutant knot pair yield manifolds of the same Gromov norm. Moreover, it is well-known that Conway mutation preserves the homeomorphism type of the branched double covering. In light of this, it is natural to ask whether $\Sigma_2(K)$ is homeomorphic to $\Sigma_2(K^\tau)$; however, this is not the case. We verify this by investigating the pair of genus two mutant knots in Figure~\ref{fig:mutantpair}, which we call $K_0$ and $K_0^\tau$ and which are known as $14^n_{22185}$ and $14^n_{22589}$ in Knotscape notation.

\begin{figure}
	\centerline{
		\psfragscanon
		\psfrag{k}{$14^n_{22185}$}
		\psfrag{kt}{$14^n_{22589}$}
		\includegraphics[width=9cm]{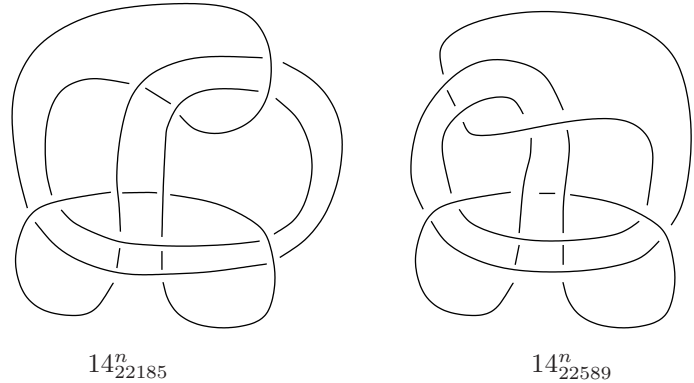}
		\psfragscanoff
		}
 	\caption{The genus two mutant pair $K_0=14^n_{22185}$ and $K_0^\tau=14^n_{22589}$.}
    	\label{fig:mutantpair}
\end{figure}

\begin{proposition}
\label{prop:nothomeomorphic}
	The branched double covers of $K_0$ and $K_0^\tau$ are not homeomorphic.
\end{proposition}
\begin{proof}
This is a fact which can be checked by computing the geodesic length spectra of $\Sigma_2(K_0)$ and $\Sigma_2(K_0^\tau)$ in SnapPy~\cite{SnapPy} with the following code snippet.\\

\centerline{
	\fbox{
		\begin{scriptsize}
\begin{lstlisting}
>> M1=Manifold("14n22185.tri"); M2=Manifold("14n22589.tri")
>> M1.dehn_fill((2,0),0); M2.dehn_fill((2,0),0)
>> M1.covers(2,cover_type="cyclic"); M2.covers(2,cover_type="cyclic")

>> M1.length_spectrum(cutoff=1.5)
mult length                           topology      parity
1    (0.618708509882-0.915396961493j) mirrored arc  orientation-preserving
1    (1.02046533287-2.87373908997j)   mirrored arc  orientation-preserving
1    (1.19267652219-1.97573028631j)   circle        orientation-preserving
1    (1.2943687184-0.108601853389j)   mirrored arc  orientation-preserving
1    (1.4180061001+1.77458043688j)    circle        orientation-preserving

>> M2.length_spectrum(cutoff=1.5)
mult length                           topology      parity
1    (0.61977975736+1.04574145952j)   mirrored arc  orientation-preserving
1    (0.946415249278+3.02707626124j)  mirrored arc  orientation-preserving
1    (1.07345426322+2.11448221051j)   circle        orientation-preserving
1    (1.2943687184-0.108601853389j)   mirrored arc  orientation-preserving
\end{lstlisting}
		\end{scriptsize}
	}
}
The complex length spectrum of a compact hyperbolic three-orbifold $M$ is the collection of all complex lengths of closed geodesics in $M$ counted with their multiplicities (Chapter 12 of~\cite{Alan:Arithmetic}). SnapPy demonstrates that the complex length spectra of $\Sigma_2(K)$ and $\Sigma_2(K^\tau)$ bounded above are different, therefore these manifolds are not isospectral, and therefore not isometric. Mostow rigidity says that the geometry of a finite-volume hyperbolic three-manifold is unique, therefore $\Sigma_2(K)$ and $\Sigma_2(K^\tau)$ are not homeomorphic.
\end{proof}
\begin{corollary}
\label{cor:notmutants}
	The genus two mutant pair $K_0$ and $K_0^\tau$ are not Conway mutants.
\end{corollary}
\begin{proof}
	Since Conway mutants have homeomorphic branched double covers, this follows directly from Proposition~\ref{prop:nothomeomorphic}.
\end{proof}
We will continue to explore the pair $14^n_{22185}$ and $14^n_{22589}$. As genus two mutants, they share all of the properties mentioned in Theorems~\ref{thm:AlexanderJones} and~\ref{thm:Volume}. Moreover, $14^n_{22185}$ and $14^n_{22589}$ are also shown in~\cite{DGST:Behavior} to have the same HOMFLY-PT and Kauffman polynomials, although in general these polynomials are known to distinguish larger examples of genus two mutant knots~\cite{DGST:Behavior}. Just as a subtler hyperbolic invariant was required to distinguish their branched double covers, we require a subtler quantum invariant to distinguish the knot pair. The categorified invariants $\hfk$ and $\kh$ do the trick.
\begin{theorem}
\label{thm:differenthomologies}
The genus two mutant knots $K_0$ and $K_0^\tau$ are distinguished by their knot Heegaard Floer homology and Khovanov homology, as well as by their $\delta$-graded versions.
\end{theorem}
See Table~\ref{table:hfkandkh}. Khovanov homology with $\Z$ coefficients was computed in~\cite{DGST:Behavior} using KhoHo~\cite{Shumakovitch:Program}. Here, we include Khovanov homology with rational coefficients computed with the Mathematica program JavaKH-v2~\cite{GM:Program}. Since $\hfk$ is known to detect Conway mutation \cite{OS:GenusBounds}, it is not surprising that knot Floer homology can distinguish genus two mutant pairs. Nonetheless, the knot Floer groups $\hfk(K_0)$ and $\hfk(K^\tau_0)$ have been computed using the Python program of Droz~\cite{Droz:Program}. The key observation is that although both knot Floer homology and Khovanov homology distinguish the genus two mutants as bigraded vector spaces, in both cases the pairs are indistinguishable as ungraded objects.
\begin{center}
\begin{table}
	\begin{tabular}{|c|}
		\hline
		\begin{tabular}{cc}
			& \\
				$\bgroup\arraycolsep=3pt
					\begin{array}{|r|rrrrr|}\hline
						\multicolumn{6}{|c|}{\hfk(K_0) } \\ \hline
						&-2&-1&0&1&2 \\ \hline
						3& & & & &  \F \\
						2& & & & \F^2& \F \\
						1& & &\F^2 & \F^2&  \\
						0& & \F^2&\F^3 & &  \\
						-1&\F & \F^2& & &  \\
						-2& \F& & & &  \\ \hline
						\multicolumn{6}{|c|}{\dim =17}\\ \hline
					\end{array}
				\egroup$
			& \hspace{0.5cm}
				$\bgroup\arraycolsep=3pt
					\begin{array}{|r|rrr|}\hline
						 \multicolumn{4}{|c|}{\hfk(K_0^\tau) }  \\ \hline
						&  -1& 0& 1  \\ \hline
						1& & & \F^2\\
						0& & \F^5& \F^2\\
						-1& \F^2&\F^4 & \\
						-2& \F^2& &  \\ \hline
						 \multicolumn{4}{|c|}{\dim =17}\\ \hline
					\end{array}
				\egroup$	\\
		\end{tabular}\\	
		\begin{tabular}{cc}
			& \\
				$\bgroup\arraycolsep=3pt
					\begin{array}{|r|rrrrr|r|}\hline
						\multicolumn{7}{|c|}{\delta-\text{graded } \hfk(K_0) } \\ \hline
						&-2&-1&0&1&2 & \dim \\ \hline
						a-m=-1&\F & \F^2& \F^2& \F^2& \F& 8  \\
						a-m=0&\F & \F^2& \F^3& \F^2& \F& 9 \\ \hline
						\multicolumn{7}{|c|}{\dim =17}\\ \hline
					\end{array}
				\egroup$
			& \hspace{0.5cm}
				$\bgroup\arraycolsep=3pt
					\begin{array}{|r|rrr|r|}\hline
						 \multicolumn{5}{|c|}{\delta-\text{graded }\hfk(K_0^\tau) }  \\ \hline
						&  -1& 0& 1  & \dim \\ \hline
						a-m=0& \F^2& \F^5& \F^2 & 9\\
						a-m=+1& \F^2&\F^4 & \F^2 & 8\\ \hline
						 \multicolumn{5}{|c|}{\dim =17}\\ \hline
					\end{array}
				\egroup$	\\
			&
		\end{tabular} \\
		\begin{tabular}{|cc|c|}
			\hline
$\kh(K_0;\Q) = $ & $1_{\underline{13}}^{\underline{7}}\, 1_{\underline{9}}^{\underline{6}}\, 1_{\underline{7}}^{\underline{4}}\, 1_{\underline{7}}^{\underline{3}}\, 1_{\underline{3}}^{\underline{3}}\, 1_{\underline{5}}^{\underline{2}}\, 1_{\underline{3}}^{\underline{2}}\, 1_{\underline{3}}^{\underline{1}}\, 1_{\underline{1}}^{\underline{1}}\, 1_{\underline{3}}^{0}\, 2_{\underline{1}}^{0}\, 2_{1}^{0}\, 2_{1}^{1}\, 1_{3}^{1}\, 1_{1}^{2}\, 1_{3}^{2}\, 1_{5}^{2}\, 1_{3}^{3}\, 1_{5}^{3}\, 1_{7}^{3}\, 1_{7}^{4}\, 1_{7}^{5}\, 1_{11}^{6}$ & 
$\dim =26$\\
		\hline
			 $\kh(K^{\tau}_0;\Q) =$ & $1_{\underline{13}}^{\underline{7}}\, 1_{\underline{9}}^{\underline{6}}\, 1_{\underline{9}}^{\underline{5}}\, 1_{\underline{9}}^{\underline{4}}\, 1_{\underline{7}}^{\underline{4}}\, 1_{\underline{5}}^{\underline{4}}\, 1_{\underline{7}}^{\underline{3}}\, 1_{\underline{5}}^{\underline{3}}\, 1_{\underline{3}}^{\underline{3}}\, 1_{\underline{5}}^{\underline{2}}\, 2_{\underline{3}}^{\underline{2}}\, 1_{\underline{3}}^{\underline{1}}\, 1_{\underline{1}}^{\underline{1}}\, 1_{1}^{\underline{1}}\, 2_{\underline{1}}^{0}\, 2_{1}^{0}\, 1_{1}^{1}\, 1_{3}^{1}\, 1_{1}^{2}\, 1_{5}^{2}\, 1_{5}^{3}\, 1_{7}^{5}\, 1_{11}^{6}$ & $\dim=26$\\ 
			\hline
		\end{tabular}
		\\
		\begin{tabular}{cc}
		& \\
			$\bgroup\arraycolsep=3pt
				\begin{array}{|l|l|}\hline
					\multicolumn{2}{|l|}{\delta-\text{graded } \kh(K_0) } \\ \hline
					q-2i=-3& 4  \\
					q-2i=-1& 11  \\
					q-2i=1& 9  \\
					q-2i=3& 2 \\ \hline
				\end{array}
			\egroup$
		& \hspace{0.5cm}
			$\bgroup\arraycolsep=3pt
				\begin{array}{|l|l|}\hline
					 \multicolumn{2}{|l|}{\delta-\text{graded } Kh(K_0^\tau) }  \\ \hline
					q-2i=-3& 2\\
					q-2i=-1& 9\\
					q-2i=1& 11 \\
					q-2i=3& 4 \\ \hline
				\end{array}
			\egroup$	
	\\
		\end{tabular} \\	
		\\
	\hline
	\end{tabular}		
	\caption{Knot Floer groups are displayed with Maslov grading on the vertical axis and Alexander grading on the horizontal axis. Computation~\cite{Droz:Program} also confirms that $\hfk(K_0)\cong\hfk(K_1)$ and $\hfk(K^\tau_0)\cong\hfk(K^\tau_1)$. For Khovanov homology, $\mathbf{R}^i_j$ denotes Khovanov groups in homological grading $i$ and quantum grading $j$ with dimension $\mathbf{R}$. The underline denotes negative gradings. This notation originated in \cite{BarNatan:OnKhovanov}.}\label{table:hfkandkh}	
\end{table}
\end{center}
\begin{figure}[h]
	\centerline{
		\includegraphics[width=5cm]{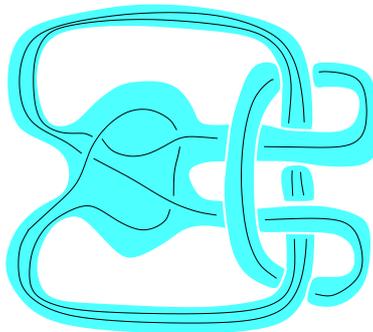}
		}
	\caption{The surface of mutation for all $K_n$. Note the surface bounds a handlebody.}  
    	\label{fig:mutationbody}
\end{figure}
\begin{figure}[h]
	\centering
	\subfloat[Oriented skein triple of $K_n$, $K_{n-2}$ and $\mathcal{U}$.]{
		\psfrag{L+}{$K_n$}
		\psfrag{L-}{$K_{n-2}$}
		\psfrag{L0}{$\mathcal{U}$}
		\label{fig:orientedskein}
		\includegraphics[width=7cm]{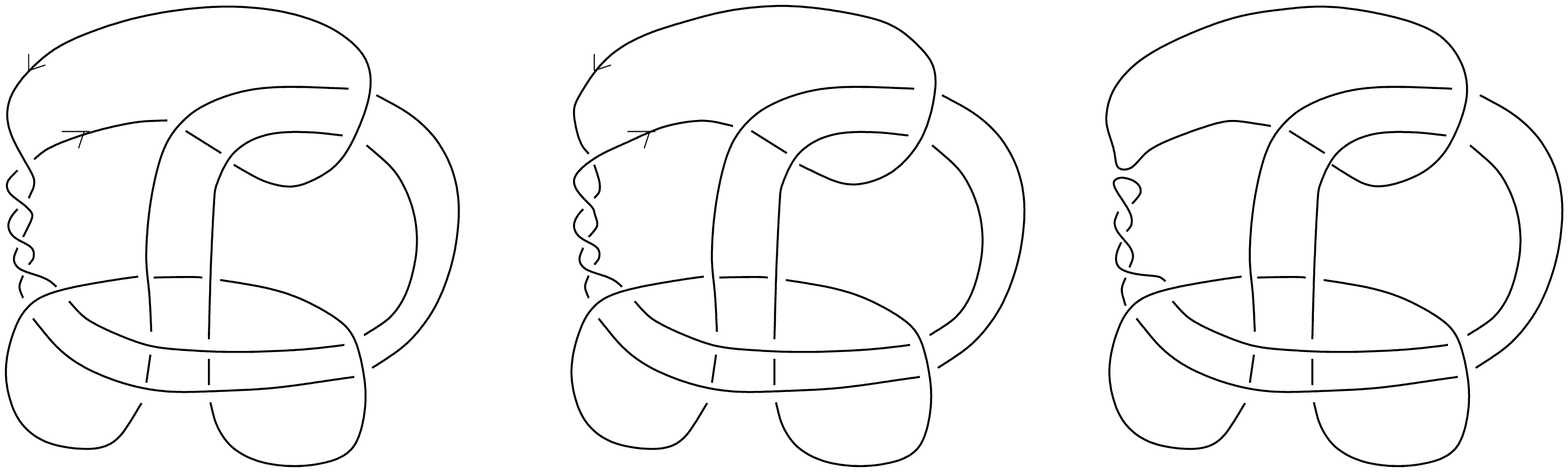}
		} 
	\;\;\;\;\;\;\;\;\;\subfloat[Unoriented skein triple of $K_n$, $K_{n-1}$ and $\mathcal{U}$.]{
		\psfrag{L+}{$K_n$}
		\psfrag{Li}{$K_{n-1}$}
		\psfrag{L0}{$\mathcal{U}$}
		\label{fig:unorientedskein}
		\includegraphics[width=7cm]{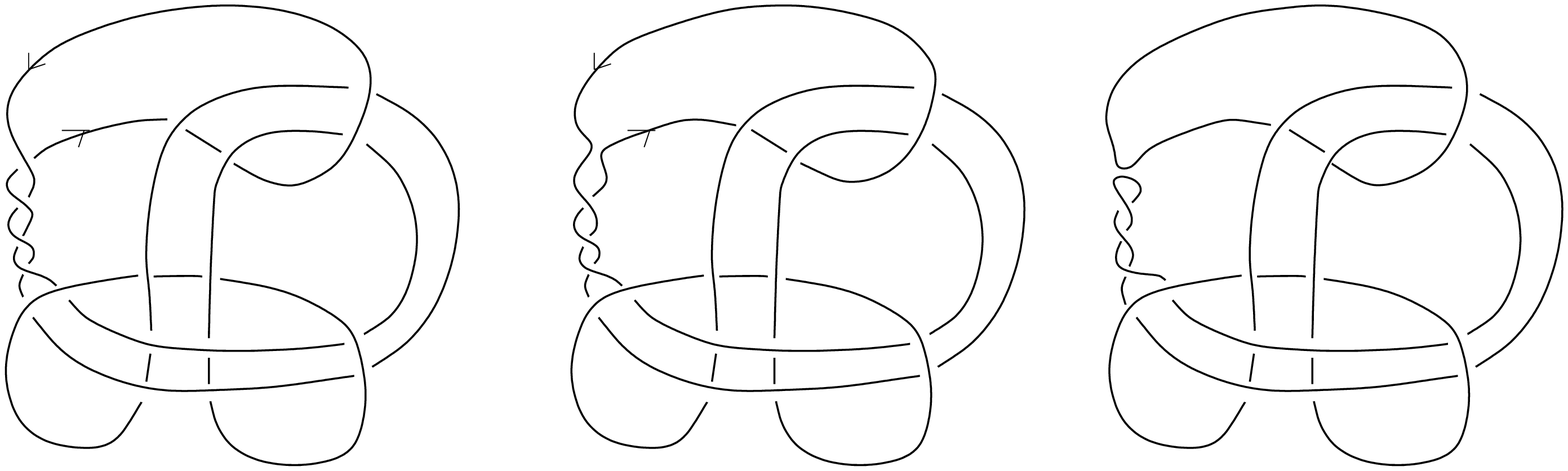}
		} 
	\caption{Oriented and unoriented skein triples.}\label{fig:skeintriples}
\end{figure}

We will derive an infinite family of knots from the pair $14^n_{22185}$ and $14^n_{22589}$. Notice that each of these can be formed as the band sum of a two-component unlink. Let us call $14^n_{22185}$ and $14^n_{22589}$ by $K_0$ and $K^\tau_0$, respectively. By adding $n$ half-twists with positive crossings to the bands of $K_0$ and $K^\tau_0$, as in Figure~\ref{fig:skeintriples}, we obtain knots $K_n$ and $K^\tau_n$. It is visibly clear that $K^\tau_n$ is the genus two mutant of $K_n$ by the same surface of mutation relating $K_0$ and $K^\tau_0$, illustrated in Figure~\ref{fig:mutationbody}.

Observe that by resolving a crossing in the twisted band, $K_n$ and $K_{n-2}$ fit into an oriented skein triple $(L_+,L_-,L_0)$ with $L_0$ equal to the two-component unlink $\mathcal{U}$ for all integers $n>1$. Moreover, $K_n$ and $K_{n-1}$ fit into an unoriented skein triple, again with third term the unlink. $K^\tau_n, K^\tau_{n-1} ,K^\tau_{n-2}$ and $\mathcal{U}$ fit into these same oriented and unoriented skein triples. 
\begin{figure}[h]
	\centerline{
		\includegraphics[width=8cm]{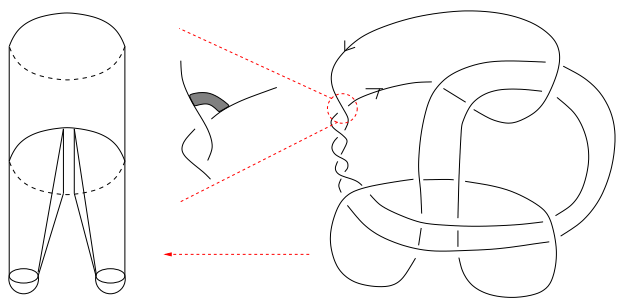}
		}
	\caption{A smooth cobordism illustrating that $K_n$ is slice. }  
   	\label{fig:knotcobordism}
\end{figure}
\begin{lemma}
\label{lemma:tauands}
	The \os~$\tau$ invariant and Rasmussen $s$ invariant vanish for all $K_n$ and $K_n^\tau$.
\end{lemma}
\begin{proof}
The knots $K_n$ and $K_n^\tau$ are formed from the band sum of a two-component unlink. In general, if $K$ is any such knot, then $K$ is smoothly slice. This is a standard fact (see for example~\cite[p. 86]{Lickorish}), and the slicing disk is illustrated in Figure~\ref{fig:knotcobordism}. \os~define the smooth concordance invariant $\tau(K)\in \Z$ in~\cite[Corollary 1.3]{OS:FourBallGenus} and Rasmussen defines a smooth concordance invariant $s(K)\in2\Z$ in~\cite[Theorem 1]{Rasmussen:Khovanov}. Both $\tau(K)$ and $s(K)$ provide lower bounds on the four-ball genus. 
\[
	|\tau(K)| \leq g_*(K) \text{  \;\;\; and \;\;\;  } |s(K)| \leq 2g_*(K).
\]
Since all of our knots are slice, we immediately obtain $\tau=s=0$.
\end{proof}
\section{Knot Floer Homology}
\label{sec:Floer}
Knot Floer homology is a powerful invariant of oriented knots and links in an oriented three manifold $Y$, developed independently by \os~\cite{OS:KnotInvariants} and Rasmussen~\cite{Rasmussen:Thesis}. We tersely paraphrase~\os's construction of the invariant for knots from~\cite{OS:KnotInvariants}, and refer the reader to~\cite{OS:KnotInvariants} for details of the construction. 
\subsection{Background from knot Floer homology}
To a knot $K\subset S^3$ is associated a doubly pointed Heegaard diagram $(\Sigma, \boldsymbol\alpha,\boldsymbol\beta,z,w)$. The data of the Heegaard diagram gives rise to chain complexes ($\cfkm(K), \partial^-)$ and $(\cfk(K), \widehat{\partial})$.  These complexes come equipped with a bigrading $(M,A)$, where $M$ denotes Maslov grading and $A$ denotes Alexander grading. $\cfkm(K)$ is an $\F_2[U]$ module, where the action of $U$ reduces $A$ by one and $M$ by two. The differentials $\partial^-$ and $\widehat{\partial}$ preserve $A$ and reduce $M$ by one. The homology groups $\hfkm(K)$ and $\hfk(K)$ are invariants of $K$. 

We will require the following theorem of~\os~specialized to the case where $L_+$ and $L_-$ are knots, which we state without proof.
\begin{theorem}~\cite[Theorem 1.1]{OS:Squence}
  \label{thm:skein}
Let $L_+$, $L_-$ and $L_0$ be three oriented links, which differ at a single crossing as indicated by the notation. Then, if $L_+$ and $L_-$ are knots, there is a long exact sequence
 \[
    \cdots\longrightarrow \hfkm_m(L_+,a)\stackrel{f^-}{\longrightarrow}\hfkm_{m}(L_-,a) \stackrel{g^-}{\longrightarrow} H_{m-1} \left( \frac{\cflm(L_0)}{U_1-U_2},a \right) \stackrel{h^-}{\longrightarrow}\hfkm_{m-1}(L_+,a)\longrightarrow \cdots
  \]
\end{theorem}

We remark that the skein exact sequence of Theorem~\ref{thm:skein} is derived from a mapping cone construction. Indeed, \os~show in~\cite[Theorem 3.1]{OS:Squence} that there is a chain map $f:\cfkm(L_+)\rightarrow \cfkm(L_-)$ whose mapping cone is quasi-isomorphic to the mapping cone of the chain map $U_1 - U_2 : \cflm(L_0)\rightarrow \cflm(L_0)$, which is in turn quasi-isomorphic to the complex $\cflm(L_0)/U_1 - U_2$. 
The maps in the diagram appearing in~\cite[Section 3.1]{OS:Squence} which determine the quasi-isomorphism from the cone of $f$ to the cone of $U_1 - U_2$ are $U$-equivariant. 
The map $f^-$ appearing in the sequence above is the map induced on homology by $f$. The maps $g^-$ and $h^-$ are induced by inclusions and projections of the mapping cone of $f$ along with the quasi-isomorphism. Therefore the long exact sequence is $U$-equivariant.
\begin{lemma}
\label{lemma:unlink}
Let $\mathcal{U}$ be the two-component unlink in $S^3$. $\mathcal{U}$ corresponds with the unknot $\widetilde{\mathcal{U}}\subset S^2\times S^1$, whose knot Floer homology is
\begin{eqnarray}
	\hfk(S^3,\mathcal{U}) \cong {\F_2}_{\begin{tiny}\begin{array}{l}m=0 \\ a=0\end{array}\end{tiny} }\oplus {\F_2}_{\begin{tiny}\begin{array}{l}m=-1 \\ a=0\end{array}\end{tiny}} \\
	H_*\left(\frac{\cflm(\mathcal{U})}{U_1-U_2} \right) \cong \hfk(S^3,\mathcal{U})\otimes_{\F_2} \F_2[U]
\end{eqnarray}
where in the module $\F_2[U]$, the action of $U$ drops the Maslov grading by two and the Alexander grading by one.
\end{lemma}
\begin{proof}
A Heegaard diagram for $\widetilde{\mathcal{U}}\subset S^2\times S^1$ can be constructed by taking a genus one splitting of $S^2\times S^1$ with two curves, $\alpha$ and $\beta$, intersecting in two points $\x$ and $\y$. Place basepoints $z$ and $w$ inside the annular region such that $\x$ is connected to $\y$ by two disks. Since it is a genus one splitting we count only $\phi$ correspnding to domains that are disks. As an application of the Riemann mapping theorem, $\#\widehat{\mathcal{M}}(\phi)=\pm1$ for each such $\phi$. Therefore the differential is zero in both $\cfk(S^2\times S^1, \widetilde{\mathcal{U}})$ and $\cfkm(S^2\times S^1, \widetilde{\mathcal{U}})$. The relative grading difference is evident from the diagram and pinned down by the observation that the $\mathcal{U}\subset S^3$ fits into a skein exact sequence (Theorem~\ref{thm:skein}) with the unknot. 
\end{proof}

\subsection{Knot Floer homology proof}	
The main objective of this section is to show that each knot in the family $\{K_n\}$ has knot Floer homology isomorphic to $\hfk(K_0)$, and that each knot in the family $\{K_n^\tau\}$ has knot Floer homology isomorphic to $\hfk(K^\tau_0)$. Similar computations generating knots with isomorphic knot homologies occur in the work of the second author~\cite{Starkston}, Watson~\cite{Watson:IdenticalKH} and Greene and Watson~\cite{GW:Turaev}, to name a few. Theorem~\ref{thm:floergroups} is a special case of an observation originally due to Hedden. It will soon appear as part of a more general result of Hedden and Watson in~\cite{Hedden:Botany}. We include a proof only for the sake of completeness and the benefit of the reader.

\begin{theorem}~\cite{Hedden:Botany}
\label{thm:floergroups}
Let $K$ be a knot in $S^3$ formed from the band sum of a two-component unlink, and let $\{K_n\}$ denote the family of knots obtained by adding $n$ half-twists with positive crossings to the band. For all $m, a \in \Z$ and $n\geq2\in\Z$, $\hfkm_m(K_{n},a) \cong \hfkm_m(K_{n-2},a)$.
\end{theorem}
\begin{proof}
The proof is by induction on $n$. Just as with the specific families of knots described above, $K_n$ fits into the skein triple $(K_n, K_{n-2}, \mathcal{U})$. Theorem~\ref{thm:skein} applied to the skein triple gives a long exact sequence
\[
	\cdots\rightarrow \hfkm_m(K_n,a)\stackrel{f^-}{\longrightarrow}\hfkm_{m}(K_{n-2},a) \stackrel{g^-}{\longrightarrow}H_{m-1} \left( \frac{\cflm(\mathcal{U})}{U_1-U_2},a \right) \stackrel{h^-}{\longrightarrow}\hfkm_{m-1}(K_n,a)\rightarrow\cdots .
\]

We will use this sequence in conjunction with information coming from the $\tau$ invariant. By Lemma~\ref{lemma:tauands}, $\tau(K_n)=0$ $\forall n$. Because we are working with $\hfkm(K)$, we will use the definition of $\tau$ appearing in~\cite[Appendix]{OS:Legendrian}, where $m(K)$ denotes the mirror of $K$.
\[ 
	\tau(m(K)) = \max \{a\;|\;\exists \; \xi\in\hfkm(K,a) \text{ such that } U^d\xi\neq0\text{ for all integers } d\geq 0 \}.
\]
Moreoever, for a homogeneous element $\xi\in\hfkm(K,\tau(m(K))$ such that $U^d\xi\neq0\;\forall d\geq 0$, the Maslov grading of $\xi$ is given by $m=2\tau(m(K))$. This fact can be verified by following the argument given in~\cite[Appendix]{OS:Legendrian}, keeping careful track of the bigrading shifts at each step.
Since $\tau(K_n)=0$, we have the additional fact that $\tau(K_n)=\tau(m(K_n))$. 

The non-torsion summand of $\hfkm(K_n)$ is generated by an element $\xi_n$ with maximal bigrading $(2\tau(m(K)), \tau(m(K))$, which in this case is $(0,0)$. 
The third term $H_* \left( \frac{\cflm(L_0)}{U_1-U_2},0 \right)$ of the skein triple corresponds with the two-component unlink and is freely generated over $\F_2[U]$ by elements $z$ and $z'$ in bigradings $(0,0)$ and $(-1,0)$. Since $\hfkm(\mathcal{U})$ is supported entirely in bigradings $(-2d,-d)$ and $(-2d-1,-d)$ the long exact sequence immediately supplies isomorphisms $\hfkm_m(K_n,a)\cong \hfkm_m(K_{n-2},a)$ whenever $a=-d\leq0$ and $|m-2a| > 1$ or when $a>0$. The $U$-equivariant long exact sequence for the remaining case is displayed below, parameterized by $d\geq0$.\\

\centerline{
	\xymatrix @R-1.5pc @C-1.25pc{ 
	0 \ar[r] & \hfkm_{1-2d}(K_n,-d) \ar[r]^{f^-} & \hfkm_{1-2d}(K_{n-2},-d) \ar[r]^>>>>>{g^-} & {\F_2}_{\{-2d,-d\}} \ar[r]^{h^-} \ar@{} [d] |{\rin}&  \hfkm_{-2d}(K_n,-d) \ar[r]^<<<<<{i^-} \ar@{} [d] |{\rin} & \\
	& & & U^d\cdot z \ar@{|->}[r]  &U^d\cdot \xi_n +\eta &&& \\
	& \hfkm_{-2d}(K_{n-2},-d) \ar[r]^>>>>{j^-} \ar@{} [d] |{\rin} & {\F_2}_{\{-1-2d,-d\}}  \ar[r]^>>>>{k^-} \ar@{} [d] |{\rin} & \hfkm_{-1-2d}(K_n,-d) \ar[r]^{\ell^-} & \hfkm_{-1-2d}(K_{n-2},-d) \ar[r] & 0 \\
	 & U^d\cdot \xi_{n-2} \ar@{|->}[r]  & U^d\cdot  z'
	}
}	
In the diagram above, equivariance of the long exact sequence with respect to the action of $U$ implies that $U^d\cdot z$ cannot be in the image of any $\F_2[U]$-torsion element. Since $\hfkm_{1-2d}(K_{n-2},-d)$ is torsion, $U^d\cdot z$ is not in the image of $g^-$, and the map $g^-=0$. Exactness implies that $f^-$ is an isomorphism, and also that $h^-$ is an injection. Since the map $h^-$ is degree preserving, $U^d\cdot z$ maps to a non-torsion element $U^d\cdot \xi_n +\eta \in \hfkm_{-2d}(K,-d)$, where $\eta$ is $\F_2[U]$-torsion.  By exactness, $U^d\cdot \xi_n+\eta \in\Ker i^-$. Because the non-torsion summand gets mapped to zero by $i^-$, $U^d\cdot \xi_{n-2}$, which is also non-torsion, is not in the image of $i^-$. By exactness, $U^d \cdot \xi_{n-2}\not\in\Ker j^-$ and $U^d\cdot \xi_{n-2}$ must map to $U^d\cdot z'$. Exactness implies that $k^-=0$ and $\ell^-$ is an isomorphism. What remains is an isomorphism of torsion submodules at $i^-$. Hence, for all $(m,a)$, $\hfkm_m(K_n,a)\cong \hfkm_m(K_{n-2},a)$.
\end{proof}
\begin{corollary}
Let $\{K_n\}$ and $\{K_n^\tau\}$ denote the infinite family of knots derived from $14^n_{22185}$ and $14^n_{22589}$. Then
	\begin{eqnarray*}
		\hfk_m(K_n,a) &  \cong & \hfk_m(K_0,a)\\
		\hfk_m(K^\tau_n,a) & \cong & \hfk_m(K^\tau_0,a).
	\end{eqnarray*}
\end{corollary}
\begin{proof}
Once a suitable base case has been established, then the result follows from relating $\hfkm(K_n)$, $\hfkm(K_{n-2})$, $\hfk(K_n)$ and $\hfk(K_{n-2})$ by the five lemma. There are four distinct families in our investigation, with base cases $K_0, K_1, K^\tau_0$ and $K_1^\tau$, for even and odd values of $n$. The hat-version $\hfk$ of each has been verified computationally with the program of Droz~\cite{Droz:Program}. $\hfk(K_1)$ and $\hfk(K^\tau_1)$ have been found to be isomorphic with $\hfk(K_0)$ and $\hfk(K_0^\tau)$, respectively (see Table~\ref{table:hfkandkh}).
%
\end{proof}

This verifies that $\{K_n\}$, $n\in\Z^+$, is an infinite family of knots admitting a distinct genus two mutant of the same total dimension in knot Floer homology.

\section{Khovanov Homology}
\label{sec:Khovanov}

Khovanov homology is a bigraded homology knot invariant introduced in \cite{Khovanov:Categorification}. The chain complex and differential of the homology theory are computed combinatorially from a knot diagram using the cube of smooth resolutions of the crossings. See \cite{BarNatan:OnKhovanov} for an introduction to the theory. Here, we compute the Khovanov homology of $K_n$ and $K_n^{\tau}$ over rational coefficients. While our computation of Heegaard Floer homology was over coefficients in $\F_2$, we need to work over $\Q$ to obtain the corresponding results in Khovanov homology. This is for two reasons. First, Rasmussen's invariant and Lee's spectral sequence are only applicable to Khovanov homology with rational coefficients, and we require these tools for the computation. Furthermore, Khovanov homology over $\F_2$ coefficients is significantly weaker at distinguishing mutants in the following sense. Bloom and Wehrli independently proved that Khovanov homology over $\F_2$ is invariant under Conway mutation in \cite{Bloom}, \cite{Wehrli:Mutation}. While these pairs are not Conway mutants, we can compute that $K_0$ and $K_0^{\tau}$ have the same $\F_2$-Khovanov homology (though we have not proven this for the infinite family). The goal of this section is to provide an infinite family of genus $2$ mutants where the bigraded rational Khovanov homology distinguishes between the knot and its mutant, whereas the total dimension of the Khovanov homology is invariant under the mutation. Our main result in this section is the following theorem.

\begin{theorem}
The Khovanov homology with rational coefficients for $K_n$ respectively $K_n^{\tau}$, for $n\geq 8$ is described by the following sequences of numbers. Here $\mathbf{R}^i_j$ denotes that the Khovanov homology in homological grading $i$ and quantum grading $j$ has dimension $\mathbf{R}$. (This notation originated in \cite{BarNatan:OnKhovanov}.)
\begin{eqnarray*}
Kh(K_n)&=& \one^0_{-1} \one^0_1\, \one^{n-7}_{2n-13} \one^{n-6}_{2n-9} \one^{n-4}_{2n-7} \one^{n-3}_{2n-7} \one^{n-3}_{2n-3}  \one^{n-2}_{2n-5} \one^{n-2}_{2n-3} \one^{n-1}_{2n-3} \one^{n-1}_{2n-1}\one^{n}_{2n-3}\one^{n}_{2n-1} \one^{n}_{2n+1}\\
&&\mathbf{2}^{n+1}_{2n+1} \one^{n+1}_{2n+3} \one^{n+2}_{2n+1} \one^{n+2}_{2n+3} \one^{n+2}_{2n+5}
\one^{n+3}_{2n+3} \one^{n+3}_{2n+5} \one^{n+3}_{2n+7} \one^{n+4}_{2n+7} \one^{n+5}_{2n+7} \one^{n+6}_{2n+11}
\end{eqnarray*}

\begin{eqnarray*}
Kh(K_n^{\tau})&=& \one^0_{-1} \one^0_1\, \one^{n-7}_{2n-13} \one^{n-6}_{2n-9} \one^{n-5}_{2n-9} \one^{n-4}_{2n-9} \one^{n-4}_{2n-7}  \one^{n-4}_{2n-5} \one^{n-3}_{2n-7} \one^{n-3}_{2n-5} \one^{n-3}_{2n-3}\one^{n-2}_{2n-5}\mathbf{2}^{n-2}_{2n-3} \one^{n-1}_{2n-3}\\
&& \one^{n-1}_{2n-1} \one^{n-1}_{2n+1} \one^{n}_{2n-1} \one^{n}_{2n+1} \one^{n+1}_{2n+1}
\one^{n+1}_{2n+3} \one^{n+2}_{2n+1} \one^{n+2}_{2n+5} \one^{n+3}_{2n+5} \one^{n+5}_{2n+7} \one^{n+6}_{2n+11}
\end{eqnarray*}
\label{KhovanovComputation}
\end{theorem}

The key aspect of this computation to note for the proof is that as $n$ increases by $1$, in all but the first two terms the homological grading increases by $1$ and the quantum grading increases by $2$. The first part of the proof will justify the computation for all but the first two terms. The second part of the proof justifies the computation of the first two terms. Before we give the proof of the computation, the following corollary highlights the relevant conclusions.

\begin{corollary} For all $n\geq 0$,
$$Kh(K_n)\not\cong Kh(K_n^{\tau})$$
as bigraded groups, and
$$Kh^{\delta}(K_n)\not\cong Kh^{\delta}(K_n^{\tau})$$
however
$$\dim(Kh(K_n))=\dim(Kh(K_n^{\tau}))=26.$$
\end{corollary}

\begin{proof}[Proof of corollary]
For $n\geq 8$ it is clear from the theorem that the bigraded Khovanov homology over $\Q$ of $K_n$ and $K_n^{\tau}$ differ. For example $K_n$ has dimension zero in homological grading $n-5$, quantum grading $2n-9$ while $K_n^{\tau}$ has dimension $1$ in that grading.

The $\delta$ graded groups can be easily computed from the theorem. The $\delta$-gradings are supported in $\delta=-3,-1,1,3$. For any value of $n$, $Kh^{\delta}(K_n)$ agrees $Kh^{\delta}(K_0)$ and $Kh^{\delta}(K_n^\tau)$ agrees with $Kh^{\delta}(K_0^\tau)$, as given in table \ref{table:hfkandkh}. In particular $Kh^{\delta}$ distinguishes $K_n$ from $K_n^\tau$.

The total dimension of the Khovanov homology in each case is $26$, and can be computed by summing the dimensions over all bidegrees.

For the finitely many cases where $0\leq n\leq 7$ this result has been computationally verified using Green and Morrison's program JavaKh-v2 \cite{GM:Program}. 
\end{proof}

\begin{proof}[Proof of theorem \ref{KhovanovComputation}]

The method of computing Khovanov homology we use here was previously used in \cite{Starkston} to find the Khovanov homology of $(p,-p,q)$ pretzel knots. The reader may refer to that paper or the above cited sources for further background and detail.

There is no difference in the proof for $K_n$ versus $K_n^{\tau}$. We will write $K_n$ throughout the proof, but all statements in the proof hold for $K_n^{\tau}$ as well.

There is a long exact sequence whose terms are given by the unnormalized Khovanov homology of a knot diagram and its $0$ and $1$ resolutions at a particular crossing. The unnormalized Khovanov homology is an invariant of a specific diagram, not of a particular knot. It is given by taking the homology of the appropriate direct sum in the cube of resolutions before making the overall grading shifts. Let $n_+$ denote the number of positive crossings in a diagram and $n_-$ the number of negative crossings. Let $[ \cdot ]$ denote a shift in the homological grading and $\{ \cdot \}$ denote a shift in the quantum grading such that $\Qsub{q}\{k\}=\Qsub{q+k}$ and such that $Kh(K)[k]$ has an isomorphic copy of $Kh^i(K)$ in homological grading $i+k$ for each $i$. \footnote{There is some discrepancy in the literature regarding the notation for grading shifts. The notation in this paper agrees with that of Bar-Natan's introduction \cite{BarNatan:OnKhovanov}, though it is the opposite of that used in Khovanov's original paper \cite{Khovanov:Categorification}. Negating all signs relating to grading shifts will give Khovanov's original notation.} 

Let $\unkh(D)$ denote the unnormalized Khovanov homology of a knot diagram $D$. Then
	$$Kh(D)=\unkh(D)[-n_-]\{n_+-2n_-\}.$$
If $D$ is a diagram of a knot, $D_0$ is the diagram where one crossing is replaced by its $0$-resolution and $D_1$ is the diagram where that crossing is replaced by its $1$-resolution. Then, we have the following long exact sequence (whose maps preserve the $q$-grading)

\begin{equation}
	\cdots \goto \unkh^{i-1}(D_1)\{1\} \goto \unkh^{i}(D) \goto \unkh^i(D_0) \goto \unkh^i(D_1)\{1\} \goto \cdots.
\label{LES}
\end{equation}

Let $D,D_0$ and $D_1$ be the diagrams for $K_n$ and its resolutions $\mathcal{U}$ and $K_{n-1}$ as shown in Figure \ref{fig:unorientedskein}. Observe that $D_0$ is a diagram for the two component unlink $\mathcal{U}$ with $6+n$ positive crossings and $7$ negative crossings. $D_1$ is a diagram for $K_{n-1}$ with $6+n$ positive crossings and $7$ negative crossings and $D$ is a diagram for $K_n$ with $7+n$ positive crossings and $7$ negative crossings. Therefore we have the following identifications
\begin{eqnarray*}
	\unkh(D_1)[-7]\{n-8\} &= &Kh(K_{n-1})\\
	\unkh(D_0)[-7]\{n-8\} &=& Kh(\mathcal{U})\\
	\unkh(D)[-7]\{n-7\}&=&Kh(K_n).
\end{eqnarray*}

Note that the Khovanov homology of the two component unlink is $Kh^0(\mathcal{U})=\Qsub{-2}\oplus \Qsub{0}^2\oplus \Qsub{2}$ and $Kh^i(\mathcal{U})=0$ for $i\neq 0$. After applying appropriate shifts we obtain $\unkh(D_0)$. We will inductively assume the computation in the theorem holds for $K_{n-1}$. The base case is established by computing $Kh(K_8)$ using the JavaKh-v2 program \cite{GM:Program}. Applying the appropriate shifts from above we thus get the value for $\unkh(D_1)$. Plugging this into the long exact sequence (\ref{LES}) gives the following exact sequences
\begin{equation}
	0\goto Kh^{i-8}(K_{n-1})\{8-n\}\{1\} \goto Kh^{i-7}(K_n)\{7-n\} \goto 0
\label{isomexact}
\end{equation}
for $i\neq 7,8$, and
\begin{eqnarray*}
	0\goto Kh^{-1}(K_{n-1})\{9-n\} \goto Kh^0(K_n)\{7-n\} \goto \Qsub{6-n}\oplus \Qsub{8-n}^2\oplus \Qsub{10-n} \goto \\
	\goto Kh^0(K_{n-1})\{9-n\} \goto Kh^1(K_n)\{7-n\} \goto 0 
\end{eqnarray*}
which by the inductive hypothesis is the same as
\begin{subequations}
	\begin{align}
		0\goto 0 \goto Kh^0(K_n)\{7-n\} \goto \Qsub{6-n}\oplus \Qsub{8-n}^2\oplus \Qsub{10-n} \goto \Qsub{8-n}\oplus \Qsub{10-n} \goto \tag{7} \label{exactsequence} \\
		\goto Kh^1(K_n)\{7-n\}\goto 0. \notag
	\end{align}
\end{subequations}

Exactness of line (\ref{isomexact}) yields isomorphisms
	$$Kh^{j-1}(K_{n-1})\{2\} \isom Kh^j(K_n)$$
for all $j\neq 0,1$. Inspecting the way the formula for $Kh(K_n)$ in the theorem depends on $n$, one can see that the inductive hypothesis verifies the computation for $Kh^j(K_n)$ for $j\neq 0,1$.

Exactness of line (\ref{exactsequence}) gives a few possibilities. Analyzing the sequence we must have
\begin{eqnarray*}
	Kh^0(K_n)&=&\Qsub{-1}\oplus \Qsub{1}^{1+a} \oplus \Qsub{3}^b\\
	Kh^1(K_n)&=&\Qsub{1}^a\oplus \Qsub{3}^b\\
\text{ where } a,b\in \{0,1\}.
\end{eqnarray*}

Now we use the fact that $s(K_n)$ vanishes by Lemma \ref{lemma:tauands}. Since $s(K_n)=0$, the spectral sequence given by Lee in \cite{Lee:Endomorphism} converges to two copies of $\Q$, each in homological grading $0$, with one in quantum grading $-1$ and the other in quantum grading $1$, as proven by Rasmussen in \cite{Rasmussen:Khovanov}. Note that the $r^{th}$ differential goes up $1$ and over $r$, because of an indexing that differs from the standard indexing for a spectral sequence induced by a filtration. (See the note in section 3.1 of \cite{Starkston} for further explanation). Let $d_r^{p,q}$ denote the differential on the $r^{th}$ page from $E_r^{p,q}$ to $E_r^{p+1,q+r}$ in Lee's spectral sequence. Here $p$ is the coordinate for the homological grading shown on the vertical axis and $q$ is the coordinate for the quantum grading shown on the horizontal axis.

\begin{table}
\resizebox{6.2in}{!}{
\begin{tabular}{|r||c|c|c|c|c|c|c|c|c|c|c|c|c|c|c|c|c|c|}
\hline
n+6&&&&&&&&&&&&&&&&&1\\\hline
n+5&&&&&&&&&&&&&&&1&&\\\hline
n+4&&&&&&&&&&&&&&&1&&\\\hline
n+3&&&&&&&&&&&&&1&1&1&&\\\hline
n+2&&&&&&&&&&&&1&1&1&&&\\\hline
n+1&&&&&&&&&&&&2&1&&&&\\\hline
n&&&&&&&&&&1&1&1&&&&&\\\hline
n-1&&&&&&&&&&1&1&&&&&&\\\hline
n-2&&&&&&&&&1&1&&&&&&&\\\hline
n-3&&&&&&&&1&&1&&&&&&&\\\hline
n-4&&&&&&&&1&&&&&&&&&\\\hline
n-5&&&&&&&&&&&&&&&&&\\\hline
n-6&&&&&&&1&&&&&&&&&&\\\hline
n-7&&&&&1&&&&&&&&&&&&\\\hline
\vdots&&&&&&&&&&&&&&&&&\\\hline
1&&a&b&&&&&&&&&&&&&&\\\hline
0&1&1+a&b&&&&&&&&&&&&&&\\\hline
 &-1&1&3&$\cdots$&m-13&m-11&m-9&m-7&m-5&m-3&m-1&m+1&m+3&m+5&m+7&m+9&m+11\\\hline
\end{tabular}
}
\caption{Here $m=-2n$. When $a=b=0$ this table gives the $\Q$-dimensions of the Khovanov homology of $K_n$ with homological grading on the vertical axis and quantum grading on the horizontal axis. This is the $E_1$ page of Lee's spectral sequence.}
\label{table:knotKhovanov}
\end{table}
\begin{table}
\resizebox{6.2in}{!}{
\begin{tabular}{|r||c|c|c|c|c|c|c|c|c|c|c|c|c|c|c|c|c|}
\hline
n+6&&&&&&&&&&&&&&&&&1\\\hline
n+5&&&&&&&&&&&&&&&1&&\\\hline
n+4&&&&&&&&&&&&&&&&&\\\hline
n+3&&&&&&&&&&&&&&1&&&\\\hline
n+2&&&&&&&&&&&&1&&1&&&\\\hline
n+1&&&&&&&&&&&&1&1&&&&\\\hline
n&&&&&&&&&&&1&1&&&&&\\\hline
n-1&&&&&&&&&&1&1&1&&&&&\\\hline
n-2&&&&&&&&&1&2&&&&&&&\\\hline
n-3&&&&&&&&1&1&1&&&&&&&\\\hline
n-4&&&&&&&1&1&1&&&&&&&&\\\hline
n-5&&&&&&&1&&&&&&&&&&\\\hline
n-6&&&&&&&1&&&&&&&&&&\\\hline
n-7&&&&&1&&&&&&&&&&&&\\\hline
\vdots&&&&&&&&&&&&&&&&&\\\hline
1&&a&b&&&&&&&&&&&&&&\\\hline
0&1&1+a&b&&&&&&&&&&&&&&\\\hline
 &-1&1&3&$\cdots$&m-13&m-11&m-9&m-7&m-5&m-3&m-1&m+1&m+3&m+5&m+7&m+9&m+11\\\hline
\end{tabular}
}
\caption{Here $m=-2n$. When $a=b=0$ this table gives the $\Q$-dimensions of the Khovanov homology of $K_n^{\tau}$ with homological grading on the vertical axis and quantum grading on the horizontal axis. This is the $E_1$ page of Lee's spectral sequence.}
\label{table:knotmKhovanov}
\end{table}

See Tables \ref{table:knotKhovanov} and \ref{table:knotmKhovanov} for the $E_1$ page on which the following analysis is carried out. In order to preserve one copy of $\Qsub{-1}$ and one copy of $\Qsub{1}$ in the $0^{th}$ homological grading we must have $d_r^{0,-1} = 0$ and $d_r^{0,1}$ acting trivially on one copy of $\Q$ for every $r$.

We may computationally verify another base case where $n=9$ and then assume $n\geq10$.  By the above inductive results, we know that $Kh^2(K_n)=0$ when $n\geq 10$. Therefore, $d_r^{1,1}=0$ for all $r\geq 1$. Thus, if $a\neq 0$, an additional copy of $\Q$ will survive in $E_{\infty}^{1,1}$ since it cannot be in the image of any $d_r$ for $r>0$. This contradicts Lee's result that there can only be two copies of $\Q$ on the $E_{\infty}$ page. Therefore $a=0$ and $d_r^{0,1}=0$ for all $r\geq 1$. Because the row corresponding to the first homological grading has zeros in quantum gradings greater than $3$, $d_r^{0,3}=0$ for all $r\geq 1$. Therefore, if $b\neq 0$, an additional copy of $\Q$ will survive in $E_{\infty}^{0,3}$, again contradicting Lee's result. Therefore $a=b=0$, and the Khovanov homology of $K_n$ and $K_n^{\tau}$ is as stated in the theorem.

\end{proof}

\section{Observation And Speculation}
\label{sec:Observation}
The families of knots which we have employed in this paper are all non-alternating slice knots, and in particular, are formed from the band sum of a two-component unlink. There are other infinite families of slice knots for which these computational techniques using skein exact sequences and concordance invariants work. For example, Hedden and Watson~\cite{Hedden:Botany} prove that there are infinitely many knots with isomorphic Floer groups in any given concordance class, whereas Greene and Watson~\cite{GW:Turaev} have worked with the Kanenobu knots. Certain pretzel knots (see~\cite{Starkston}) also share this property. Nor is the non-alternating status of these knots a coincidence; in fact there can only be finitely many alternating knots of a given knot Heegaard Floer homology type.

\begin{proposition}
Let $K$ be an alternating knot. There are only finitely many other alternating knots with knot Floer homology isomorphic to $\hfk(K)$ as bigraded groups.
\end{proposition}
\begin{proof}
Suppose to the contrary that $K$ belongs to an infinite family $\{K_n\}_{n\in\Z}$ of alternating knots sharing the same knot Floer groups. Since $\hfk(K_n)\cong\hfk(K)$ and knot Floer homology categorizes the Alexander polynomial,
\[
	\det(K_n) = |\Delta_{K_n}(-1)| =  |\Delta_{K}(-1)| = \det(K)
\]
for all $n$. Each knot $K_n$ admits a reduced alternating diagram $D_n$ with crossing number $c(D_n)$. The Bankwitz Theorem implies that $c(K_n) \leq \det(K_n)$. However, there are only finitely many knots of a given crossing number, and in particular $c(K_n)$ grows arbitrarily large with $n$, which contradicts that $c(K_n) \leq \det(K)$.
\end{proof}
This fact leads to the interesting open question of whether there are infinitely many quasi-alternating knots of a given knot Floer type. Greene formulates an even stronger conjecture in~\cite{Greene:Homologically}, and proves the cases where $\det(L)=1,2$ or $3$.
\begin{conjecture}[Conjecture 3.1 of~\cite{Greene:Homologically}]
There exist only finitely-many quasi-alternating links with a given determinant.
\end{conjecture}

In Section~\ref{sec:Khovanov}, we mention that $K_0$ and $K_0^\tau$ have the same Khovanov homology with $\F_2$ coefficients. In fact, $(K_0,K_0^\tau)$ is one of five pairs of genus two mutants appearing in~\cite{DGST:Behavior}, none of which can be distinguished by Khovanov homology over $\F_2$. Bloom and Wehrli~\cite{Bloom},\cite{Wehrli:Mutation} have shown that Khovanov homology with $\F_2$ coefficients is invariant under component-preserving Conway mutation. This leads to another unanswered question.
\begin{question}
\label{question:z2Khovanov}
Is Khovanov homology with $\F_2$ coefficients invariant under genus two mutation?
\end{question}
Because there is a spectral sequence relating the reduced Khovanov homology of $L$ over $\F_2$ to the Heegaard Floer homology of the branched double cover of $-L$, this raises another natural question. 
\begin{question}
If $K$ and $K^\tau$ are genus two mutant knots, is $\rank \hf(\Sigma_2(K)) = \rank \hf(\Sigma_2(K^\tau))$?
\end{question}

Genus two mutation provides a method for producing closely related knots and links, but more generally it is an operation on three manifolds. This yields yet another unanswered question:
\begin{conjecture}
Let $M$ be a closed, oriented three-manifold with an embedded genus two surface $F$. If $M^\tau$ is the genus two mutant of $M$, then
\[
	\rank \hf(M) = \rank \hf(M^\tau)
\]
\end{conjecture}

The question of whether the total rank is preserved under Conway mutation remains an interesting problem. The evidence that we offer above suggests that the total ranks of knot Floer homology and Khovanov homology are also preserved by genus two mutation. Because genus two mutation along a surface which does not bound a handlebody does not correspond in an obvious way to an operation on a knot diagram, a combinatorial proof of this general statement may be difficult to obtain.
\section{Acknowledgments}
We would like to thank our advisors, Cameron Gordon and Robert Gompf, for their guidance and support. We would also like to thank John Luecke, Matthew Hedden, and Cagri Karakurt for helpful conversations. We are also grateful to Adam Levine for pointing out an error in a previous version, and to the anonymous referee who made numerous helpful comments. The first author was partially supported by the NSF RTG under grant no. DMS-0636643. The second author was supported by the NSF Graduate Research Fellowship under grant no. DGE-1110007.
\bibliographystyle{plain}
\bibliography{bibliography}

\end{document}